\documentclass[a4paper, 12pt]{article}
\usepackage[sumlimits]{amsmath}
\usepackage{amsthm}
\usepackage{amssymb}
\usepackage{amsfonts}
\usepackage{graphicx}
\usepackage[latin1]{inputenc}
\usepackage{enumerate}
\usepackage{subfigure}

\newcommand{\bbC}{\mathbb{C}}
\newcommand{\bbZ}{\mathbb{Z}}

\newcommand{\epi}{\twoheadrightarrow}

\newcommand{\defeq}{\mathrel{\mathop:}=}

\renewcommand{\phi}{\varphi}
\renewcommand{\epsilon}{\varepsilon}

\newcommand{\scr}[1]{\ensuremath{\mathcal{#1}}}

\newcommand{\calD}{\scr D}
\newcommand{\calL}{\scr L}
\newcommand{\calH}{\scr H}
\newcommand{\catO}{\scr O}

\newcommand{\calR}{\scr R}

\newcommand{\enva}[1]{\scr U(#1)}

\newcommand{\deriv}[1]{\scr D^b(#1)}
\newcommand{\homot}[1]{\scr K^b(#1)}

\newcommand{\roots}{R}

\newtheorem{proposition}{Proposition}[section]
\newtheorem{lemma}[proposition]{Lemma}
\newtheorem{theorem}[proposition]{Theorem}
\newtheorem{corollary}[proposition]{Corollary}

\theoremstyle{definition}

\DeclareMathOperator{\Supp}{Supp}
\DeclareMathOperator{\Ann}{Ann}

\DeclareMathOperator{\Hom}{Hom}
\DeclareMathOperator{\Ext}{Ext}
\DeclareMathOperator{\mindeg}{mindeg}

\DeclareMathOperator{\Ind}{Ind}
\DeclareMathOperator{\Res}{Res}

\DeclareMathOperator{\pr}{pr}
\def\clap#1{\hbox to 0pt{\hss#1\hss}}

\renewcommand{\frak}{\mathfrak}
\newcommand{\luszta}{\mathbf{a}}
\newcommand{\inv}{{-1}}

\newcommand{\wcirc}{{w_\circ^{\phantom{I}}}}
\newcommand{\wIcirc}{{w_\circ^I}}

\begin{document}

\title{Kostant's problem and parabolic subgroups}

\author{Johan K\aa hrstr\"om}

\maketitle

%%%%%%%%%%%%%%%%%%%%%%%%%%%%%%%%%%%%%%%%%%%%%%%%%%%%%%%%%%%%%%%%%%%%%%%%%%%%%%%
%%%%%%%%%%%%%%%%%%%%%%%%%%%%%%%%%%%%%%%%%%%%%%%%%%%%%%%%%%%%%%%%%%%%%%%%%%%%%%%
%%%%%%%%%%%%%%%%%%%%%%%%%%%%%%%%%%%%%%%%%%%%%%%%%%%%%%%%%%%%%%%%%%%%%%%%%%%%%%%
%%%%%%%%%%%%%%%%%%%%%%%%%%%%%%%%%%%%%%%%%%%%%%%%%%%%%%%%%%%%%%%%%%%%%%%%%%%%%%%
%%%%%%%%%%%%%%%%%%%%%%%%%%%%%%%%%%%%%%%%%%%%%%%%%%%%%%%%%%%%%%%%%%%%%%%%%%%%%%%
%%%%%%%%%%%%%%%%%%%%%%%%%%%%%%%%%%%%%%%%%%%%%%%%%%%%%%%%%%%%%%%%%%%%%%%%%%%%%%%
%%%%%%%%%%%%%%%%%%%%%%%%%%%%%%%%%%%%%%%%%%%%%%%%%%%%%%%%%%%%%%%%%%%%%%%%%%%%%%%

\abstract{%
	Let $\frak g$ be a finite dimensional complex semi-simple
	Lie algebra with Weyl group $W$
	and simple reflections $S$. For $I\subseteq S$ let
	$\frak g_I$ be the corresponding semi-simple subalgebra of $\frak g$.
	Denote by $W_I$ the Weyl group of $\frak g_I$ and let
	$\wcirc$ and $\wIcirc$ be the longest elements of $W$ and $W_I$,
	respectively. In this paper we show that the answer to Kostant's problem,
	i.e. whether the universal enveloping algebra surjects onto the space
	of all ad-finite linear transformations of a given module,
	is the same for the simple highest weight $\frak g_I$-module $L_I(x)$
	of highest weight $x\cdot 0$, $x\in W_I$, as the answer for the
	simple highest weight $\frak g$-module $L(x\wIcirc\wcirc)$
	of highest weight $x\wIcirc\wcirc\cdot 0$.
	We also give a new description of the
	unique quasi-simple quotient of the Verma module $\Delta(e)$ with the
	same annihilator as $L(y)$, $y\in W$.
	}

\section{Introduction}

Let $\frak g=\frak n^-\oplus\frak h\oplus\frak n$ be a
finite dimensional complex semi-simple Lie algebra
with a chosen triangular decomposition, and let $\enva{\frak g}$ be
its universal enveloping algebra. For two $\frak g$-modules $M$ and $N$,
the space $\Hom_\bbC(M, N)$ of linear maps from $M$ to $N$ has
a $\enva{\frak g}$-bimodule structure in the natural way
(see for example~\cite[Kapitel~6]{jantzenEinhullende}), and hence
a $\frak g$-module structure via the adjoint action. The $\frak g$-submodule
of $\Hom_\bbC(M, N)$ consisting of all locally finite elements
is in fact a $\enva{\frak g}$-sub-bimodule, which we denote by
$\calL(M, N)$. As $\enva{\frak g}$ itself is locally finite
under the adjoint action, we have a natural homomorphism of
$\enva{\frak g}$ into $\calL(M, M)$ for every $\frak g$-module $M$, whose
kernel is the annihilator $\Ann M$ of $M$ in $\enva{\frak g}$.
The question raised by Kostant (see for
example~\cite[6.10]{conze}, \cite{josephKostant})
is: for which $\frak g$-modules $M$ is the natural inclusion
\[
	\enva{\frak g}/\Ann M\hookrightarrow\calL(M, M)
\]
a \emph{surjection}.

This is in general a difficult question, and the
answer is not even known for simple highest weight modules.
It is known to have the positive answer for Verma modules (\cite[6.9]{conze} for
simple Verma modules, generalized in \cite[6.4]{josephKostant} for
the general case) and for all quotients
of dominant Verma modules~\cite[6.9]{jantzenEinhullende}.
For semi-simple Lie algebras having roots of different length,
examples of simple highest weight modules
where the answer is negative were found early (see for example~\cite[6.5]{conzeduflo},
\cite[9.5]{josephKostant}). More recently, many examples have also been
found in type $A$ (see~\cite{mazorchukstroppelWedderburn}
and~\cite{kahrstrommazorchuk}).
The answer to Kostant's problem
is a valuable tool for example when determining Goldie rank ratios
(see~\cite{josephGoldieRankOne, josephKostantGoldieRank, josephSumRule}),
and in the study of generalized Verma modules
(see~\cite{milicicsoergel, khomenkomazorchukInduced,
mazorchukstroppelInducedCellModules}).

In this note we investigate how the answer to this question for certain
simple highest weight $\frak g$-modules relates
to the answer for modules of semi-simple subalgebras of $\frak g$.
More precisely,
let $W$ be the Weyl group of $\frak g$, with simple reflections $S$,
determined by the triangular decoposition.
For a subset $I\subseteq S$, let $W_I$ denote the parabolic subgroup
of $W$ generated by $I$, denote by $\frak g_I$ the corresponding
semi-simple subalgebra of $\frak g$, and let $\wcirc$ and $\wIcirc$ denote the longest
elements of $W$ and $W_I$.
For $x\in W$, let $L(x)$ denote the simple
highest weight $\frak g$-modules with highest weight $x\cdot 0$ (see next
section for precise definition),
and similarly, for $x\in W_I$, let $L_I(x)$ denote the simple highest
weight $\frak g_I$-module with highest weight $x\cdot 0$. The main
result of this paper is the following theorem,
which generalizes previous results by
Conze-Berline and Duflo~\cite[2.12 and~6.3]{conzeduflo},
later generalized by Gabber and Joseph~\cite[4.4]{gabberjoseph} (the case when $x=e$),
and Mazorchuk~\cite[Theorem~1]{mazorchukTwistedApproach} (the case when $x$ is a
simple reflection).

\begin{theorem}\label{thm:main}
	Let $x\in W_I$. Then Kostant's problem has the positive answer for
	$L_I(x)$ if and only if Kostant's problem has the positive answer for
	$L(x\wIcirc \wcirc)$.
\end{theorem}

The idea of the proof is as follows. For each $x\in W_I$, there is a unique
quotient $D$ of the dominant Verma module $\Delta_I(e)$ satisfying
$\Ann D = \Ann L_I(x)$. Since Kostant' problem has the positive answer for $D$,
as it is a quotient of a dominant Verma module, we see that
Kostant's problem has the positive answer for $L_I(x)$ if and only if
\begin{equation}\label{eq:intone}
	\calL_I(D, D)\cong\calL_I\bigl(L_I(x), L_I(x)\bigr)
\end{equation}
(where the index $I$ is used to emphasize that objects are defined
with respect to $\frak g_I$ as opposed to $\frak g$).
We show that we can `lift' this situation
by parabolic induction, i.e.
there exists a $\frak g$-module $D'$ for which the answer to Kostant's problem
is positive, and such that
\[
	\calL(D', D')\cong\calL\bigl(L(x\wIcirc\wcirc), L(x\wIcirc\wcirc)\bigr)
\]
holds if and only if~\eqref{eq:intone} holds.

In Section~\ref{sec:D} we give an alternative description of the so-called
\emph{quasi-simple} quotients the dominant Verma module, originally described
in~\cite[Section~5]{josephDixmiersProblem}, which are used as an important
tool in the proof of Theorem~\ref{thm:main}. Finally, in Section~\ref{sec:sl6}
we apply Theorem~\ref{thm:main} to get some new answers to Kostant's problem
for the Lie algebra $\frak{sl}_6$.

\medskip

\noindent{\bf Acknowledgements.}
The author thanks V.~Mazorchuck for fruitful comments, suggestions and discussions.

\section{Notation and preliminaries}

The subset $I$ of $S$ determines a parabolic subalgebra
$\frak p_I$ of $\frak g$, containing $\frak g_I$. The triangular
decomposition of $\frak g$ induces a triangular decomposition
$\frak g_I = \frak n_I^-\oplus\frak h_I\oplus\frak n_I$.
Let $\frak u_I$ be the nilradical of $\frak p_I$, and let $\frak z_I$ be the
orthogonal complement of $\frak h_I$ in $\frak h$ with respect
to the Killing form. We thus have the following decompositions,
\[
\frak h = \frak h_I \oplus \frak z_I,\text{ and }
\frak p_I = \frak g_I\oplus \frak z_I\oplus \frak u_I.
\]

The Weyl group $W$ of $\frak g$ acts on $\frak h^*$ in the natural way $w\lambda$,
but in this setting it is more convenient to consider the so-called `dot action',
given by
\[
	w\cdot \lambda \defeq w(\lambda + \rho) - \rho,
\]
where $\rho$ is the half sum of the positive roots. Similarly we
have both the standard action and dot action of $W_I$ on $\frak h_I^*$.
 
Let $\catO$ denote the BGG category (see for example~\cite{bgg, humphreys}),
and let $\catO_0$ denote
the principal block of $\catO$, i.e. the full subcategory of $\catO$
consisting of modules that are
annihilated by some power of the maximal ideal of the
center of $\enva{\frak g}$ which annihilates the trivial module.
The simple modules of $\catO_0$ are the simple highest weight modules $L(w)$
of highest weight $w\cdot 0$, where $w$ runs over $W$.
We denote the Verma module with simple head $L(w)$ by $\Delta(w)$,
and the projective cover of $L(w)$ by $P(w)$. Finally, for $w\in W$
we denote
by $\theta_w$ the indecomposable projective functor on $\catO_0$
(see~\cite{bg}) satisfying
\[
	\theta_w\Delta(e) = P(w).
\]
The corresponding objects for $\frak g_I$ are denoted $\catO^I$, $L_I(w)$,
$\calL_I$, etc.

For a subalgebra $\frak a$ of $\frak g$ (here $\frak a$
will be either $\frak h_I$ or $\frak z_I$), a module $M\in\catO$, and
$\lambda\in\frak a^*$, let
\[
	M_\lambda\defeq \bigl\{\,m\in M\,\big\vert\,xm=\lambda(x)m
		\text{ for all $x$ in $\frak a$}\,\bigr\},
\]
and define the support of $M$ with respect to $\frak a$ as
\[
	\Supp_{\frak a} M\defeq
	\bigl\{\,\lambda\in\frak a^*\,\big\vert\,M_\lambda\neq 0\,\bigr\}.
\]

\section{Parabolic induction}

For $\lambda\in\frak z_I^*$, we define the
\emph{induction functor} from $\catO^I$ to $\catO$ by
\[
	\Ind_\lambda M \defeq \enva{\frak g}\otimes_{\enva{\frak p_I}}M^\lambda,
\]
where $M^\lambda$ is the $\frak p_I$-module obtained from $M$
by letting $\frak z_I$ act by $\lambda$, and $\frak u_I$ act by $0$.
We also define the \emph{restriction functor} from $\catO$ to
$\catO^I$ by
\[
	\Res_\lambda M\defeq M_\lambda,
\]
where the action is restricted to $\frak g_I$.

\begin{lemma}\label{lem:indann}
	If $\Ann_{\enva{\frak g_I}} M=\Ann_{\enva{\frak g_I}} N$
	for two $\frak g_I$-modules $M$ and $N$, then
	$\Ann_{\enva{\frak g}}\Ind_\lambda M=\Ann_{\enva{\frak g}}\Ind_\lambda N$
	for any $\lambda\in\frak z_I^*$.
\end{lemma}

\begin{proof}
	We have
	\begin{multline*}
		\Ann_{\enva{\frak p_I}} M^\lambda =
		\bigl(\Ann_{\enva{\frak g_I}} M\bigr)\otimes\enva{\frak z_I}\otimes\enva{\frak u_I}
		+ \enva{\frak g_I}\otimes\ker\lambda\otimes\enva{\frak u_I} \\
		+ \enva{\frak g_I}\otimes\enva{\frak z_I}\otimes\enva{\frak u_I}_{>0},
	\end{multline*}
	where $\enva{\frak u_I}_{>0}$ denotes the elements of $\enva{\frak u_I}$
	of degree at least $1$. Hence
	$\Ann_{\enva{\frak p_I}} M^\lambda = \Ann_{\enva{\frak p_I}} N^\lambda$,
	so the result follows
	from~\cite[Proposition~5.1.7(ii)]{dixmier}.
\end{proof}

Let $\roots_I$ be the simple roots corresponding to $I$.
The fundamental weights of $\frak h_I^*$ dual to $\roots_I$ define a basis $B_I$ of
$\frak z_I^*$, which in turn define a partial order on $\frak z_I^*$
by $\nu\leqslant\lambda$ for $\nu, \lambda\in\frak z_I^*$
if $\lambda-\nu$ is in the non-negative span of $B_I$.
For $\lambda\in\frak z_I^*$ and $M\in\catO$, let
$M_{\not\leqslant\lambda}$ be the submodule of $M$ generated
by all $M_\nu$, $\nu\not\leqslant\lambda$, and define
\[
	M^{\leqslant\lambda} \defeq M/M_{\not\leqslant\lambda}.
\]

Generalising the situation when tensoring Verma modules with finite
dimensional modules, we get the following.

\begin{lemma}\label{lem:tensorfiltration}
	For a finite dimensional $\frak g$-module $V$, $M\in\catO^I$,
	and $\lambda\in\frak z_I^*$,
	the module $V\otimes \Ind_\lambda M$ has a filtration
	\[
		0=M_0\subset M_1\subset\cdots\subset M_k=V\otimes\Ind_\lambda M
	\]
	with
	\[
		M_i/M_{i-1}\cong\Ind_{\lambda+\mu_i}\Bigl(\bigl(\Res_{\mu_i} V\bigr)\otimes M\Bigr),
	\]
	where $\mu_1>\mu_2>\cdots>\mu_k\in\frak z_I^*$ and
	$\Supp_{\frak z_I} V = \{\mu_1, \dots, \mu_k\}$.
\end{lemma}

\begin{proof}
	Let $\mu_1$, $\dots$, $\mu_k\in\frak z_I^*$ be as in the lemma,
	let $B_1$, $\dots$, $B_k$ be bases of $\Res_{\mu_1}V$, $\dots$,
	$\Res_{\mu_k}V$, and
	let $B$ be a basis of $M$. Now define
	\[
		M_i\defeq \sum_{1\leq j\leq i}\enva{\frak g}\Bigl(B_j\otimes
			\bigl(1\otimes_{\enva{\frak p_I}}B\bigr)\Bigr).
	\]
	As in the `standard' case
	(se for instance~\cite[Satz~2.2]{jantzenEinhullende}) we find that
	each $M_i$ is $\enva{\frak u_I^-}$-free over
	\[
		\bigcup_{1\leq j\leq i}B_j\otimes\bigl(1\otimes_{\enva{\frak p_I}} B\bigr).
	\]
	In particular, as $\enva{\frak u_I^-}$-modules we have that
	\[
		M_i/M_{i-1}\cong \enva{\frak u_I^-}\Bigl(B_i\otimes\bigl(1\otimes_{\enva{\frak p_I}} B\bigr)\Bigr).
	\]
	Furthermore, it is straightforward to see that,
	as $\enva{\frak g_I}$-modules,
	\[
		\enva{\frak g_I}\Bigl(B_i\otimes\bigl(1\otimes_{\enva{\frak p_I}}B\bigr)\Bigr)
		\cong \bigl(\Res_{\mu_i}V\bigr)\otimes M,
	\]
	from which the statement follows.
\end{proof}

\begin{corollary}\label{cor:resind}
	For any $\lambda, \mu\in\frak z_I^*$, finite dimensional $\frak g$-module
	$V$, and $M\in\catO^I$, we have
	\[
		\Res_\mu(V\otimes \Ind_\lambda M)^{\leqslant\mu}\cong
		\bigl(\Res_{\mu-\lambda} V\bigr)\otimes M.
	\]
\end{corollary}

\begin{proof}
	If $\mu-\lambda\notin\Supp_{\frak z_I} V$ the result is immediate as both modules
	are zero.
	On the other hand, if $\mu-\lambda\in\Supp_{\frak z_I} V$,
	then by Lemma~\ref{lem:tensorfiltration} the module
	$(V\otimes \Ind_\lambda M)^{\leqslant \mu}$ has a submodule $M'$
	isomorphic to
	\[
		\Ind_\mu\Bigl(\bigl(\Res_{\mu-\lambda}V)\otimes M\Bigr),
	\]
	and
	\[
		\Supp_{\frak z_I}\bigl((V\otimes \Ind_\lambda M)^{\leqslant \mu}/M'\bigr)<\mu,
	\]
	from which the statement follows.
\end{proof}

We now fix $\xi\in\frak z_I^*$ to be the restriction of $\wcirc\cdot 0$
to $\frak z_I$, and let $\catO^\xi$ be the full subcategory of $\catO$
of modules satisfying $\Supp_{\frak z_I} M\leqslant\xi$.
By~\cite[Proposition~11]{mazorchukTwistedApproach},
$\Ind_\xi$ and $\Res_\xi$ induce mutually inverse equivalences
between $\catO^\xi_0$ and $\catO^I_0$, identifying $L_I(x)$ with
$L(x\wIcirc\wcirc)$ and $\Delta_I(x)$ with $\Delta(x\wIcirc\wcirc)$.
Let $\pr_0$ and $\pr^I_0$ denote the projection functors from $\catO$
to $\catO_0$ and $\catO^I$ to $\catO^I_0$, respectively.

\begin{lemma}\label{lem:indrescommute}
	For any $M\in\catO^\xi$ we have
	\[
		\Res_\xi\circ\pr_0(M)\cong\pr^I_0\circ\Res_\xi(M).
	\]
\end{lemma}

\begin{proof}
	Let $\lambda\in\frak h^*$ with $\lambda\vert_{\frak z_I}\leq\xi$.
	If $\lambda\vert_{\frak z_I}<\xi$ then
	\[
		\Res_\xi\circ\pr_0\bigl(L(\lambda)\bigr) =
		\pr^I_0\circ\Res_\xi\bigl(L(\lambda)\bigr) = 0,
	\]
	so assme $\lambda\vert_{\frak z_I}=\xi$.
	We then have that
	\[
		\Res_\xi L(\lambda)\cong L_I\bigl(\lambda\vert_{\frak h_I}\bigr).
	\]
	Furthermore, since $\lambda\vert_{\frak z_I} = (\wcirc\cdot 0)\vert_{\frak z_I}$,
	we have that
	\[
		\pr_0 L(\lambda)\cong
		\begin{cases}
			L(\lambda) & \text{ if $\lambda\in W_I\wcirc\cdot 0$, or equivalently,
					$\lambda\vert_{\frak h_I}\in W_I\cdot 0$,} \\
			0 & \text{otherwise}.
		\end{cases}
	\]
	Hence the statement follows for simple modules since
	\[
		\pr^I_0 L_I\bigl(\lambda\vert_{\frak h_I}\bigr) =
		\begin{cases}
			L_I\bigl(\lambda\vert_{\frak h_I}\bigr)
				& \text{ if $\lambda\vert_{\frak h_I}\in W_I\cdot 0$,} \\
			0 & \text{ otherwise.}
		\end{cases}
	\]
	Now let $M\in\catO^\xi$, and let $M_0\in\catO^\xi_0$ and $M_1\in\catO^\xi$
	be such that
	\[
		M\cong M_0\oplus M_1.
	\]
	By definition, we have
	\begin{equation}\label{eq:projeq1}
		\Res_\xi\circ\pr_0 M \cong \Res_\xi M_0.
	\end{equation}
	Let $L(\lambda)$ be a composition factor of $M_1$.
	If $\lambda\vert_{\frak z_I} < \xi$ then $\Res_\xi L(\lambda)=0$,
	and if $\lambda\vert_{\frak z_I} = \xi$ we must have
	$\lambda\vert_{\frak z_I}\notin W_I\cdot 0$, so
	$\pr^I_0\circ\Res_\xi L(\lambda)=0$. Since both restriction
	and projection are exact it follows that
	\[
		\pr^I_0\circ\Res_\xi M_1=0.
	\]
	On the other hand, since $M_0\in\catO^\xi_0$ we have
	$\Res_\xi M_0\in \catO^I_0$, so
	\[
		\pr^I_0\circ\Res_\xi M_0 \cong \Res_\xi M_0.
	\]
	Since both restriction and projection are additive, it follows that
	\[
		\pr^I_0\circ\Res_\xi M \cong \Res_\xi M_0.
	\]
	Comparing with~\eqref{eq:projeq1} yields the result.
\end{proof}

\section{Proof of Theorem~\ref{thm:main}}

We start by proving the building blocks used in the proof of Theorem~\ref{thm:main}.

\begin{proposition}\label{prop:dimequal}
	For each finite dimensional $\frak g$-module $V$ and $M, N\in\catO_0^I$, we have
	\[
		\Hom_{\frak g}\bigl(V\otimes\Ind_\xi M, \Ind_\xi N\bigr)
		\cong \Hom_{\frak g_I}\bigl(\Res_0V\otimes M, N\bigr).
	\]
\end{proposition}

\begin{proof}
	We have that
	\begin{align*}
		\Hom_{\frak g}\bigl(V\otimes \Ind_\xi M, \Ind_\xi N\bigr)
		&\cong \Hom_{\frak g}\bigl(\pr_0(V\otimes\Ind_\xi M)^{\leqslant\xi},
				\Ind_\xi N\bigr) \\
		&\cong \Hom_{\frak g_I}\bigl(\Res_{\xi}\circ\pr_0(V\otimes\Ind_\xi M)^{\leqslant\xi},
				N\bigr) \\
		&\cong \Hom_{\frak g_I}\bigl(\pr^I_0\circ\Res_{\xi}(V\otimes\Ind_\xi M)^{\leqslant\xi},
				N\bigr) \\
		&\cong \Hom_{\frak g_I}\bigl(\Res_{\xi}(V\otimes\Ind_\xi M)^{\leqslant\xi},
				N\bigr) \\
		&\cong \Hom_{\frak g_I}(\Res_0 V\otimes M, N),
	\end{align*}
	where the first isomorphism follows from the fact that
	$\Ind_\xi N\in\catO_0^\xi$, the second by the adjointness of
	$\Res_\xi$ and $\Ind_\xi$, the third by Lemma~\ref{lem:indrescommute},
	the fourth by the fact that $N\in\catO^I_0$, and the fifth by
	Corollary~\ref{cor:resind}.
\end{proof}

\begin{corollary}\label{cor:dimequal}
	For $M, N\in\catO_0^I$ we have
	\[
		\Hom_{\frak g_I}\bigl(V, \calL_I(M, M)\bigr) \cong
		\Hom_{\frak g_I}\bigl(V, \calL_I(N, N)\bigr)
	\]
	for all finite dimensional $\frak g_I$-modules V if and only if
	\[
		\Hom_{\frak g}\bigl(V', \calL\bigl(\Ind_\xi M, \Ind_\xi M\bigr)\bigr) \cong
			\Hom_{\frak g}\bigl(V', \calL\bigl(\Ind_\xi N, \Ind_\xi N\bigr)\bigr)
	\]
	for all finite dimensional $\frak g$-modules $V'$.
\end{corollary}

\begin{proof}
	For the `only if' part, by Proposition~\ref{prop:dimequal}
	and \cite[6.8~(3)]{jantzenEinhullende} we have
	\begin{align*}
		\Hom_{\frak g}\bigl(V', \calL(\Ind_{\xi}M, \Ind_{\xi}M)\bigr)
		&\cong \Hom_{\frak g}\bigl(V'\otimes \Ind_{\xi}M, \Ind_{\xi}M\bigr) \\
		&\cong \Hom_{\frak g_I}\bigl(\Res_0V'\otimes M, M\bigr) \\
		&\cong \Hom_{\frak g_I}\bigl(\Res_0V'\otimes N, N\bigr) \\
		&\cong \Hom_{\frak g_I}\bigl(\Res_0V', \calL_I(\Ind_{\xi}N, \Ind_{\xi}N)\bigr) \\
		&\cong \Hom_{\frak g}\bigl(V', \calL(\Ind_{\xi}N, \Ind_{\xi}N)\bigr)
	\end{align*}
	for all finite dimensional $\frak g$-modules $V'$.
	Similarly, for the `if' part, we find that
	\[
		\Hom_{\frak g_I}\bigl(\Res_0V', \calL_I(M, M)\bigr)
		\cong\Hom_{\frak g_I}\bigl(\Res_0V', \calL_I(N, N)\bigr)
	\]
	for all finite dimensional $\frak g$-modules $V'$. We need to show
	that this covers all relevant finite dimensional $\frak g_I$-modules.
	We first note that
	\[
		\Hom_{\frak g_I}(V\otimes M, M) \neq 0
	\]
	only if	$V_0 \neq \{0\}$, where $V_0$ denotes the $\frak h_I$-invariant subspace
	of $V$. This follows from the fact that
	\[
		\Supp_{\frak h_I}(V\otimes M)\subseteq \Supp_{\frak h_I}V+\Supp_{\frak h_I}M
	\]
	and, since $M\in\catO_I$,
	\[
		\Supp_{\frak h_I}M\subset\bbZ \roots_I,
	\]
	while, if $V_0=\{0\}$,
	\[
		\Supp_{\frak h_I}V\cap\bbZ\roots_I=\emptyset.
	\]
	On the other hand, extending the highest weight of $V$ from $\frak g_I$
	to $\frak g$ and using the classification of finite dimensional
	$\frak g$-modules (see~\cite[Theorem~7.2.6]{dixmier}
	we have that
	if $V_0\neq\{0\}$ then there is a finite dimensional $\frak g$-module
	$V'$ such that $V$ is a direct summand of $\Res_0 V'$. Now the result
	follows by induction on the dimension of $V$.
\end{proof}

The following crucial observation is due to V.~Mazorchuk.

\begin{proposition}\label{prop:botquot}
	Kostant's problem has the positive answer for any quotient of
	$\Delta(\wIcirc\wcirc)$.
\end{proposition}

\begin{proof}
	Consider a short exact sequence
	\[
		0\rightarrow X\rightarrow \Delta(\wIcirc\wcirc)\rightarrow Y\rightarrow 0.
	\]
	By \cite[Proposition 5]{mazorchukTwistedApproach}, we need to show that
	\[
		\Ext^1_{\catO}\bigl(\Delta(\wIcirc\wcirc), \theta_xX\bigr) = 0
	\]
	for all $x\in W$. Let $C_x$ and $T_x$ denote the completion functor
	and the twisting functor associated with $x\in W$, respectively,
	and let $\calR C_x$ and $\calL T_x$ denote the corresponding right and
	left derived functors.
	They satisfy
	\[
		C_x\Delta(\wcirc)\cong\Delta(x^\inv\wcirc),\text{ and }
		T_x\Delta(x^\inv\wcirc)\cong\Delta(\wcirc),
	\]
	they form mutually inverse equivalences of
	the bounded derived category $\deriv{\catO}$, and they commute
	with projective functors (all this can be found
	in~\cite{andersenstroppel} and~\cite{khomenkomazorchuk}). Hence we have
	\begin{align*}
		\Ext^1_{\catO}\bigl(\Delta(\wIcirc\wcirc), \theta_xX\bigr)
		&\cong \Hom_{\deriv{\catO}}\bigl(\Delta(\wIcirc\wcirc)[-1], \theta_xX\bigr) \\
		&\cong \Hom_{\deriv{\catO}}\bigl(\calR C_{\wIcirc}\Delta(\wcirc)[-1], \theta_xX\bigr) \\
		&\cong \Hom_{\deriv{\catO}}\bigl(\theta_{x^\inv}\Delta(\wcirc)[-1],
			\calL T_{\wIcirc}X\bigr).
	\end{align*}

	To study $\calL T_{\wIcirc}X$,
	we note that $X\in\catO^\xi_0$, and take a projective resolution
	\[
		P^\bullet\epi X:\qquad0\rightarrow P_k\rightarrow\cdots
			\rightarrow P_1\rightarrow P_0\rightarrow X\rightarrow 0,
	\]
	of $X$ in this category.
	
	In $\deriv{\catO}$ we now have $X\cong P^\bullet$.
	Since $\catO^\xi_0$ is equivalent to
	$\catO^I_0$, since all projective modules in $\catO^I_0$ have Verma flags,
	and since the equivalence maps Verma modules to Verma modules, the
	modules in $P^\bullet$ have Verma flags. Since $T_{\wIcirc}$
	is acyclic on such modules we have $\calL T_{\wIcirc}P^\bullet = T_{\wIcirc}P^\bullet$,
	and hence we have
	\begin{multline*}
		\Hom_{\deriv{\catO}}\bigl(\theta_{x^\inv}\Delta(\wcirc)[-1], \calL T_{\wIcirc}X\bigr) \\
		\cong
		\Hom_{\deriv{\catO}}\bigl(\theta_{x^\inv}\Delta(\wcirc)[-1], T_{\wIcirc}P^\bullet\bigr).
	\end{multline*}
	For $x\in W_I$, let $\tilde P(x\wIcirc\wcirc)$ denote
	the projective cover of the simple $L(x\wIcirc\wcirc)$ in
	$\catO^{\leqslant\lambda}_0$. We have that
	\[
		\Delta(\wIcirc\wcirc) = \tilde P(\wIcirc\wcirc),
	\]
	and, analogous to $\catO_0$, for each $x\in W_I$
	there is a projective functor $\tilde\theta_x$ such that
	\[
		\tilde P(x\wIcirc\wcirc)\cong \tilde\theta_x\Delta(\wIcirc\wcirc).
	\]
	Since twisting functors commute with projective functors
	we have
	\[
		T_{\wIcirc}\tilde P(x\wIcirc\wcirc)
		\cong T_{\wIcirc}\tilde\theta_x\Delta(\wIcirc\wcirc)
		\cong \tilde\theta_x T_{\wIcirc}\Delta(\wIcirc\wcirc)
		\cong \tilde\theta_x \Delta(\wcirc).
	\]
 	Since $\Delta(\wcirc)$ is a tilting module, and projective functors
	take tilting modules to tilting modules, we have that
	$T_{\wIcirc}\tilde P(x\wIcirc\wcirc)$ is a tilting module for all
	$x\in W_I$.
	In particular, $T_{\wIcirc}P^\bullet$ is a complex of tilting modules.
	Similarly, $\theta_{x^\inv}\Delta(\wcirc)$ is a tilting module,
	and hence we have
	\begin{multline*}
		\Hom_{\deriv{\catO}}\bigl(\theta_{x^\inv}\Delta(\wcirc)[-1], T_{\wIcirc}P^\bullet\bigr)\\
		\cong \Hom_{\homot{\catO}}\bigl(\theta_{x^\inv}\Delta(\wcirc)[-1],
				T_{\wIcirc}P^\bullet\bigr),
	\end{multline*}
	by~\cite[Chapter III(2), Lemma~2.1]{happel}, where $\homot{\catO}$ is
	the bounded homotopy category.
	Since $\theta_{x^\inv}\Delta(\wcirc)[-1]$ is concentrated in position $1$,
	and $\calL T_{\wIcirc}P^\bullet$ lies between position $0$ and $-k$, this last
	$\Hom$-space must be zero.
\end{proof}

We can now put the above results together to prove Theorem~\ref{thm:main}.

\begin{proof}[Proof of Theorem~\ref{thm:main}]
	By~\cite[Lemma~3.3]{josephGoldieRankOne},
	there is a (unique)
	quotient $D$ of $\Delta_I(e)$ satisfying
	$\Ann L_I(x)=\Ann D$, and Kostant's problem
	has the positive answer for $D$, since $D$ is a quotient
	of the dominant Verma module (see for example, \cite[6.9]{jantzenModuln}).
	Hence we have
	\begin{equation}\label{eq:lindinjfirst}
		\calL_I(D, D)\cong\enva{\frak g_I}/\Ann D \cong
		\enva{\frak g_I}/\Ann L_I(x)\hookrightarrow\calL_I\bigl(L_I(x), L_I(x)\bigr).
	\end{equation}
	Furthermore, since $L(x\wIcirc\wcirc)\cong \Ind_\xi L_I(x)$ we have
	\[
		\Ann L(x\wIcirc\wcirc) = \Ann\Ind_\xi D
	\]
	by Lemma~\ref{lem:indann}.
	Since $\Ind_\xi D$ is a quotient of
	$\Ind_\xi\Delta_I(\wIcirc)\cong\Delta(\wIcirc\wcirc)$, Kostant's
	problem has the positive answer for $\Ind_\xi D$ by Proposition~\ref{prop:botquot}.
	As above, we have
	\begin{equation}\label{eq:lindinj}
		\calL\bigl(\Ind_\xi D, \Ind_\xi D\bigr)
		\hookrightarrow\calL\bigl(L(x\wIcirc\wcirc), L(x\wIcirc\wcirc)\bigr).
	\end{equation}
	
	If Kostant's problem has the positive answer for $L(x)$ then the
	injection~\eqref{eq:lindinjfirst} is a bijection, so by
	Corollary~\ref{cor:dimequal} we have
	\[
		\Hom_{\frak g}\Bigl(V, \calL\bigl(\Ind_\xi D, \Ind_\xi D\bigr)\Bigr)
		\cong \Hom_{\frak g}\Bigl(V, \calL\bigl(L(x\wIcirc\wcirc), L(x\wIcirc\wcirc)\bigr)\Bigr)
	\]
	for all finite dimensional $\frak g$-modules $V$. Hence the
	injection~\eqref{eq:lindinj} is a bijection, and Kostant's problem has the
	positive answer for $L(x\wIcirc\wcirc)$. The proof of the converse is
	completely analogous.
\end{proof}

\section{Alternative description of $D$}\label{sec:D}

The module $D$ used in the proof of Theorem~\ref{thm:main} can be described as follows.
If we set $J=\Ann L(x)$, then by \cite[Lemma~3.3]{josephGoldieRankOne},
$J\Delta(e)$ is the unique submodule of $\Delta(e)$ satisfying
\[
	\Ann\bigl(\Delta(e)/J\Delta(e)\bigr) = \Ann L(x).
\]
In particular, $D\defeq \Delta(e)/J\Delta(e)$ is the unique
quotient of $\Delta(e)$ satisfying $\Ann D=\Ann L(x)$.

When beginning this work, the author used a more direct approach
to find the module $D$, inspired by ideas in~\cite{kahrstrommazorchuk}.
Although not necessary for the current exposition, the following result
is interesting in its own right. 

\begin{proposition}\label{prop:dquasi}
	Let $x\in W$. The unique quotient $D$ of $\Delta(e)$
	satisfying $\Ann D=\Ann L(x)$ is isomorphic to the image of a
	non-zero homomorphism
	\[
		\Delta(e)\rightarrow\theta_{x}L(x^\inv).
	\]
\end{proposition}

We first note that this image is uniquely defined, since
\begin{align*}
	\dim\Hom_{\frak g}\bigl(\Delta(e), \theta_{x}L(x^\inv)\bigr)
	&= \dim\Hom_{\frak g}\bigl(\theta_{x^\inv}\Delta(e), L(x^\inv)\bigr) \\
	&= \dim\Hom_{\frak g}\bigl(P(x^\inv), L(x^\inv)\bigr) \\
	&= 1.
\end{align*}
To prove Proposition~\ref{prop:dquasi} we need to recall some further theory.

The category $\catO_0$ has a $\bbZ$-graded version
$\catO_0^\bbZ$, in which the modules $L(x)$, $\Delta(x)$ and $P(x)$, for
$x\in W$, all have standard graded lifts (where their heads are
concentrated in degree zero). Furthermore, the projective functors
$\theta_x$, $x\in W$, also have graded lifts, see~\cite{stroppelTranslationFunctors}.
For $M\in\catO_0^\bbZ$
and $i\in\bbZ$,
let $M\langle i\rangle$ denote the graded module defined
by $M\langle i\rangle_j \defeq M_{j-i}$.

The Grothendieck group
of $\catO_0^\bbZ$ is isomorphic to the Hecke algebra $\calH$ of W, i.e. the
free $\bbZ[v, v^\inv]$-module over the basis $\{\,H_x\,\vert\,x\in W\,\}$, where
multiplication is given by $H_xH_y = H_{xy}$ if $\ell(xy)=\ell(x)+\ell(y)$,
and $H_sH_s = H_e + (v^\inv-v)H_s$ for simple reflections $s\in S$.
The Kazhdan-Lusztig basis is a basis of the Hecke algebra,
whose elements we denote by $\underline H_x$, which are self dual
under the duality $H\mapsto\overline{H}$ on $\calH$ given by
$\overline H_x = (H_{x^\inv})^\inv$ and $\overline v = v^\inv$.
We also have the dual Kazhdan-Lusztig basis, whose elements we denote by
$\hat{\underline H}_x$, which is dual to the Kazhdan-Lusztig basis with
respect to the symmetrising trace. We then have
\begin{align*}
	[\Delta(x)] &= H_x, \\
	[P(x)] &= \underline H_x, \\
	[L(x)] &= \hat{\underline H}_x, \\
	[\theta_x\underline{\phantom M}] &=
		\text{ right multiplication by $\underline H_x$, and} \\
	[\underline{\phantom M}\langle i\rangle] &= \text{ multiplication by $v^{-i}.$}
\end{align*}
For a review of this theory, see~\cite{mazorchukstroppelInducedCellModules},
in particular Section~3.

For $x, y\in W$ and $H\in\calH$ let $k^H_{x, y}\in\bbZ[v, v^\inv]$ be such that
\[
	\underline H_xH = \sum_{y\in W}k^H_{x, y}\underline H_y.
\]
The \emph{right preorder} on $W$ is defined by $x\leqslant_R y$ if there exists
an $H\in\calH$ with $k^H_{x, y}\neq 0$. Dually, if
$\hat k^H_{x, y}\in\bbZ[v, v^\inv]$ is such that
\[
	H\underline{\hat H}_x = \sum_{y\in W}\hat k^H_{x, y}\underline{\hat H}_y,
\]
then $x\geqslant_R y$ if and only if there exists a $H\in\calH$
with $\hat k^H_{x, y}\neq 0$ (see~\cite[5.1.16]{lusztigReductive}).
The \emph{left preorder} is defined by $x\leqslant_L y$ if
and only if $x^\inv\leqslant_R y^\inv$.
By~\cite{josephCharacteristicVariety, vogan, kl} we have the important
fact that
\[
	x\leqslant_L y \text{ if and only if } \Ann L(x)\supseteq\Ann L(y).
\]
The
equivalence classes of $\leqslant_R$ and $\leqslant_L$
are called right and left cells, respectively.

For $x, y\in W$, let $h_{x, y}\in\bbZ[v, v^\inv]$ with
\[
	\underline H_y = \sum_{x\in W} h_{x, y} H_x,
\]
and for $x, y, z\in W$, let $k_{x, y, z}\in\bbZ[v, v^\inv]$ with
\[
	\underline H_x\underline H_y =
	\sum_{z\in W}k_{x, y, z}\underline H_z.
\]
Note in particular that $\overline k_{x, y, z}=k_{x, y, z}$.
Now Lusztig's $\luszta$-function on $W$ (see~\cite{lusztigCells})
can be defined as
\[
	\luszta(x) \defeq \max_{y, z\in W}\deg k_{y, z, x}.
\]
It is constant on right cells, and in general
we have (see~\cite[1.3(1)]{lusztigCellsTwo})
\[
	\luszta(x)\leq \mindeg h_{e, x},
\]
where, for $f\in\bbZ[v, v^\inv]$,
$\mindeg f$ is the \emph{minimal degree} of $f$, i.e. the minimal element
$i\in\bbZ$ such that the coefficient of $v^i$ in $f$ is non-zero.
The \emph{Duflo set} $\calD$ (sometimes
called the set of distinguished involutions) is defined
as the set of elements $d\in W$ satisfying
\[
	\luszta(d)=\mindeg h_{e, d}.
\]
By~\cite[Proposition~1.4, Theorem~1.10]{lusztigCellsTwo}, each right
cell contains precisely one Duflo involution. Note that, by the BGG reciprocity,
we have
\[
	[\Delta(e)] = \sum_{x\in W}h_{e, x}[L(x)].
\]
Hence, given a right cell $R$ of $W$, all composition factors on the form
$L(x)$, $x\in R$ of $\Delta(e)$ occur in degree at least $\luszta(x)$,
and there is precisely one such element which occur in degree $\luszta(x)$,
namely the Duflo involution in $R$.

\begin{proof}[Proof of~Proposition~\ref{prop:dquasi}]
	Fix $x\in W$ and denote the image of a non-zero homomorphism from $\Delta(e)$
	to $\theta_xL(x^\inv)$ by $\bar D$. Since $\theta_x$ is exact, applying
	it to
	\[
		P(x^\inv)\epi L(x^\inv)
	\]
	gives
	\begin{equation}\label{eq:pontol}
		\theta_xP(x^\inv)\epi \theta_xL(x^\inv).
	\end{equation}
	
	Firstly, we have, for some $\hat k_{x^\inv, x, z}\in\bbZ[v, v^\inv]$,
	\[
		\bigl[\theta_xL(x^\inv)\bigr] = \hat{\underline H}_{x^\inv}\underline H_x
		= \sum_{z\in W}\hat k_{x^\inv, x, z}\hat{\underline H}_z
		= \sum_{z\in W}\hat k_{x^\inv, x, z}[L(z)],
	\]
	and $k_{x^\inv, x, z}\neq 0$ implies $z\leqslant_R x^\inv$
	so all composition factors of $\theta_xL(x^\inv)$ are on the form
	$L(y)$, where $y\leqslant_R x^\inv$.
	On the other hand, we have
	\[
		\bigl[\theta_xP(x^\inv)\bigr] = \underline H_{x^\inv}\underline H_x
		= \sum_{z\in W}k_{x^\inv, x, z}\underline H_z
		= \sum_{z\in W}k_{x^\inv, x, z}[P(z)],
	\]
	and $k_{x^\inv, x, z}\neq 0$ implies $z\geqslant_R x^\inv$. Hence the head of
	$\theta_xP(x^\inv)$ has only simple factors on the form $L(y)$, $y\geqslant_R x^\inv$.
	From~\eqref{eq:pontol} it follows that
	$\theta_xL(x^\inv)$ has minimal degree greater than or equal to
	$-\luszta(x^\inv)$, and that the head of $\theta_xL(x^\inv)$ has
	only simple factors on the form $L(y)$, $y\sim_R x^\inv$. Furthermore, since
	$\theta_xL(x^\inv)$ is self-dual, $\theta_xL(x^\inv)$ has maximal degree
	smaller than or equal to $\luszta(x^\inv)$, and all its simple submodules
	are on the form $L(y)$, $y\sim_Rx^\inv$.
	
	In particular, the maximal degree of $\bar D$ is bounded by $\luszta(x^\inv)$,
	and all simple submodules of $\bar D$ are on the form $L(y)$, $y\sim_Rx^\inv$.
	But the only such submodule occurring on degree $\luszta(x^\inv)$
	or smaller in $\Delta(e)$ is $L(d)$, where $d$ is the unique Duflo involution
	in the same right cell as $x^\inv$,
	occurring precisely once in degree $\luszta(x^\inv)$.
	Hence $\bar D$ has the unique simple submodule $L(d)$, and all other
	simple composition factors are on the form $L(y)$, $y<_Rd$.
	By~\cite[Proposition 6.2 (ii)]{josephKostant} it follows that
	$\Ann\bar D=\Ann L(d)$, and $\Ann L(d)=\Ann L(x)$
	as $d\sim_L x$. Since $D$ is the unique
	quotient of $\Delta(e)$ with this property, we must have $\bar D=D$.
\end{proof}

\section{Kostant's problem for $\mathfrak{sl}_6$}\label{sec:sl6}

In~\cite{kahrstrommazorchuk}, the answer to Kostant's problem was given
for all simple modules in $\catO_0$ for $\mathfrak{sl}_n$, $n\leq 5$,
and partial results were obtained for $\mathfrak{sl}_6$.
In type $A$ the answer to Kostant's problem is a left cell invariant by
\cite[Theorem~60]{mazorchukstroppelInducedCellModules}. Furthermore, since in
type $A$ there is one unique involution in each left cell, it suffices
to solve Kostant's problem for involutions. The Weyl group for $\mathfrak{sl}_6$
is $S_6$, which contains 76 involutions. For 45 of these Kostant's problem
were shown to have the positive answer, for 17 the answer was negative, and for 11
it remained unknown.

We expected that Theorem~\ref{thm:main} would answer many of these 11 unknown cases,
but it actually turned out to answer only two. The involution
$s_1s_2s_1s_5$ is in the same left cell as the element
\[
	s_1s_4\cdot \wIcirc\wcirc,
\]
where $I=\{s_1, s_2, s_3, s_4\}$. By~\cite[Corollary~21]{kahrstrommazorchuk},
Kostant's problem has the positive answer for the $\frak{sl}_5$-module
$L(s_1s_4)$, and hence by Theorem~\ref{thm:main} Kostant's problem has
the positive answer for the $\frak{sl}_6$-module $L(s_1s_2s_1s_5)$.
By symmetry of the Coxeter diagram,
Kostant's problem also has the positive answer for $L(s_1s_4s_5s_4)$.
Hence answer to Kostant's problem is still unknown for the modules
\begin{align*}
&L(s_2s_3s_4s_3s_2), &&L(s_2s_1s_4s_3s_2s_5s_4), &&L(s_1s_3s_2s_4s_3s_2s_1s_5s_4s_3),\\
&L(s_2s_1s_3s_4s_3s_2), &&L(s_1s_2s_3s_2s_4s_3s_2s_1), &&L(s_2s_1s_3s_2s_1s_4s_5s_4s_3s_2),\\
&L(s_2s_4s_3s_2s_5s_4), &&L(s_2s_3s_2s_4s_5s_4s_3s_2), &&L(s_2s_1s_3s_2s_4s_3s_2s_1s_5s_4s_3s_2).
\end{align*}

\vspace{1cm}

\noindent 
Contact: Department of Mathematics, Uppsala University, SE-751 06, Uppsala, 
SWEDEN, e-mail: {\small \tt johank@math.uu.se}, \\
web: http://www.math.uu.se/$\tilde{\hspace{1mm}}$johank/

\end{document}